\documentclass[11pt]{amsart}
\usepackage{geometry}                % See geometry.pdf to learn the layout options. There are lots.
\usepackage{graphicx}
\usepackage{amssymb}
\usepackage{amsmath}

\title[Laguerre eigenvalue density on short scales]{Convergence of the eigenvalue density for $\beta-$Laguerre ensembles on short scales}
\author{Philippe Sosoe and Percy Wong}

\newtheorem{theorem}{Theorem}
\newtheorem{proposition}{Proposition}[section]
\newtheorem{corollary}[theorem]{Corollary}
\newtheorem{lemma}[proposition]{Lemma}

\theoremstyle{definition}

\begin{document}

\begin{abstract}In this note, we prove that the normalized trace of the resolvent of the $\beta$-Laguerre ensemble eigenvalues is close to the Stieltjes transform of the Marchenko-Pastur (MP) distribution with very high probability, for values of the imaginary part greater than $m^{-1+\epsilon}$. As an immediate corollary, we obtain convergence of the one-point density to the MP law on short scales. The proof serves to illustrate some simplifications of the method introduced in our previous work \cite{sosoewong} to prove a local semi-circle law for Gaussian $\beta$-ensembles.
\end{abstract}

\maketitle 

\section{Introduction}
Consider an $m \times n$ matrix $X$,  whose entries are i.i.d. complex Gaussian random variables with mean 0 and variance $\mathbb{E}|X_{ij}|^2=1$. The $m\times m$ matrix $H=XX^*$, the star $*$ denoting the conjugate transpose, is a \emph{Wishart matrix} \cite{wishart}. It is a classical result of random matrix theory that the distribution of the eigenvalues $\lambda_1, \ldots, \lambda_m$ of $H$ is given by the density $f_2(\lambda)$ on $\mathbb{R}^m$:
\[f_2(\lambda) = Z_{2,m}^{-1}\prod_{i<j}|\lambda_i-\lambda_j|^2 \cdot \prod_{i=1}^m \lambda_i^{n-m}e^{-\beta\sum^m_{i=1}\lambda_i/2}.\] 
Here $Z_{2,m}$ is a normalization factor.
Suppose the limit $m/n\rightarrow d$ exists for $0 <d\le 1$. Then the empirical distribution
\[F_{2,m}(x) =\frac{1}{m}\cdot \sharp \{1 \le i\le m: \lambda_i \le x \}\]
 of the eigenvalues of the rescaled matrix $H/m$ converges to the Marchenko-Pastur distribution, with density 
 \[\rho_{MP,d}(x)  = \frac{1}{2d\pi x} \sqrt{(\lambda_+-x)(x-\lambda_-)} \mathbf{1}_{[\lambda_-,\lambda_+]}(x),\]
where
 \[\lambda_\pm = (1\pm \sqrt{d})^2.\]
In fact, V.A. Marchenko and L.A. Pastur \cite{marchenkopastur} showed that weak convergence of the eigenvalue distribution also holds when the entries of $X$ are i.i.d. but not necessarily Gaussian. 

A. Edelman and I. Dumitriu \cite{edelmandumitriu} have introduced an infinite family of tridiagonal random matrix models, termed $\beta$-Laguerre matrices, which generalize the Wishart model, and have the explicit eigenvalue density
\begin{equation}
\label{eq: density}
f_{\beta,a}(\lambda) = Z_{\beta,m}^{-1}\prod_{i<j}|\lambda_i-\lambda_j|^\beta \cdot \prod_{i=1}^m \lambda_i^{a-\frac{\beta}{2}(m-1)-1}e^{-\beta \sum^m_{i=1}\lambda_i/2},
\end{equation}
for $0<\beta<\infty$. Here $a> \beta(m-1)/2$ is a real parameter. $Z_{\beta,m}$ is another normalization factor. The (Gaussian) Wishart eigenvalue distribution corresponds to the Laguerre ensemble with $\beta=2$ and $a=\beta n/2$. Almost-sure convergence of the eigenvalue density to the Marchenko-Pastur distribution with parameter
\[d = \lim \frac{m\beta}{2a}<1\]
was established for the eigenvalue distributions of the $\beta$-Laguerre matrix models by the moment method in \cite{dumitriuth}; that is, for any $a<b$:
\[F_{\beta,n}(b)-F_{\beta,n}(a)\rightarrow \int_{a}^b\,\rho_{MP,d}(s)\,\mathrm{d}s,\]
almost surely as $n\rightarrow \infty$. $F_{\beta,n}(x)$ is the eigenvalue distribution function, defined as in the case $\beta=2$ above.

The object of the present note is to extend these results on convergence of the eigenvalue distribution of general $\beta$-Laguerre ensembles to \emph{short} intervals $[a_m,b_m]$ such that $b_m-a_m =O(m^{-1+\epsilon})$, for $\epsilon>0$ arbitrary. 

The main result is the following:
\begin{theorem} \label{thm1} Let $\beta>0$, and $0 <d\le 1$. Set $a=m\beta/2d$. Let $\delta,\kappa, \epsilon >0$ be positive parameters, and $E\in (\lambda_-+\kappa,\lambda_+-\kappa)$. Denote by $N_I$ the number of eigenvalues of the $\beta$-Laguerre ensemble of size $m$ with parameter $a$ in the interval $I$. For any $k>0$, there exists a constant $C_{\delta,\kappa,\epsilon,k}$ such that:
\[\mathbb{P}\left(\left|\frac{1}{m}N_{[E-m^{-1+\epsilon},E+m^{-1+\epsilon}]}-\int_{E-m^{-1+\epsilon}}^{E+m^{-1+\epsilon}}\rho_{MP,d}(x)\,\mathrm{d}x\right|\ge \delta m^{-1+\epsilon}\right) \le C_{\delta,\kappa,\epsilon,k} m^{-k}.\]
\end{theorem}
The proof of Theorem \ref{thm1} is obtained by combining Theorem \ref{thm:main} with Corollary \ref{cor:ESYY} in Section \ref{sec: inductive}. Following recent work of L. Erd\"os, B. Schlein, H.T. Yau and collaborators in the case of Wigner matrices (see for example \cite{erdosyau1}, \cite{erdosyau2}, \cite{erdosyau3}, \cite{erdosyau4}, \cite{erdosyau5}, \cite{erdosyau6}), our method is based on the study of the resolvent matrix of the Edelman-Dumitriu tridiagonal models. 

Let $M_{\beta,d}$ denote a normalized tridiagonal $\beta$-Laguerre matrix (see Section \ref{sec: expansion} for details on notation). The imaginary part of the trace of the resolvent
\[s_{\beta,d}(z)=\frac{1}{n}\operatorname{tr} (M_{\beta,d}-z)^{-1}.\]
provides an approximation for the eigenvalue distribution on scales comparable to the distance between $z$ and the spectrum of $M_{\beta,d}$. In \cite{sosoewong}, we showed how a resolvent expansion, together with an iterative argument based on the Schur complement identity could be used to derive a local version of the semi-circle law for the Gaussian $\beta$-ensembles. Tridiagonal models for these eigenvalue distributions appeared in \cite{edelmandumitriu}. In Proposition \ref{onequarterprop}, we use a resolvent expansion to show that $s_{\beta,d}(z)$ is close to the Stieltjes transform of $\rho_{MP,d}$ for values of $z$ away from the spectral edges and with $\Im z> m^{-1/4+\epsilon}$. This implies that Theorem \ref{thm1} holds for intervals $I$ of size $|I|\ge m^{-1/4+\epsilon}$. The argument in the present work is substantially simpler than the corresponding one in \cite{sosoewong}, where we proved a local convergence result on the scale $m^{-1/2+\epsilon}$ using a resolvent expansion and asymptotics for Hermite polynomials derived by the Riemann-Hilbert method. Instead of attempting to exploit the cancellation due to oscillation of the normalized Laguerre polynomials, we use a general off-diagonal resolvent estimate (see Lemma \ref{resolventest}). The computation of the limit of the normalized resolvent trace for the deterministic matrix corresponding to $\beta=\infty$ using Riemann-Hilbert asymptotics in \cite{sosoewong} has been replaced by a less involved derivation, see Section \ref{sec: infty}.

The iterative argument leading to Theorem \ref{thm:main} in Section \ref{sec: inductive} is similar to the one in \cite{sosoewong}. Note that Theorem \ref{thm:main} is deduced from Propositions \ref{prop:induct1} and \ref{prop:induct2}, without reference to Proposition \ref{onequarterprop}. In contrast, the local result on the intermediate scale $m^{-1/2+\epsilon}$ in \cite{sosoewong} was used as an input for an inductive argument to reach the scale $m^{-1+\epsilon}\ $\footnote{Modulo minor changes, the inductive argument in \cite{sosoewong} can also be used to obtain the semi-circle law down to scale $n^{-1+\epsilon}$.}. Although it can be entirely replaced by the iteration in Section \ref{sec: inductive}, we have chosen to present the argument for Proposition \ref{onequarterprop} because it provides an elementary alternative to the Schur complement approach to proving the Marchenko-Pastur law. The proof of Proposition \ref{onequarterprop} does not depend on the specific properties of the $\beta$-Laguerre ensembles other than concentration of the entries around their mean. It can be applied to fairly general tridiagonal models with independent entries to prove convergence of the eigenvalue distribution, and convergence up to some intermediate scale depending on the magnitude of the entries. We give examples of such extensions in Section \ref{sec: extensions}.

We end this section with some references to previous literature. Local versions of the Marchenko-Pastur law for the eigenvalue distribution of $XX^*$  when the entries of $X$ are independent but not necessarily Gaussian have appeared in \cite{maltsevschlein}, \cite{erdosyau4}, \cite{taovucov}. The first paper deals with the hard edge of the spectrum in case $d=1$. The use of a perturbative expansion around deterministic matrices associated to Laguerre polynomials already appears in \cite{edelmandumitriu2}, where the authors study fluctuations of the spectrum in the large $\beta$ limit. In \cite{popescu}, I. Popescu proves convergence and Gaussian fluctuations for the moments of tridiagonal matrices under a scaling assumption for the moments of the entries.  In contrast to the result in \cite{sosoewong} at the time of publication, Theorem \ref{thm1} appears to be new for general $\beta$.

\newpage 
\section{Tridiagonal models and resolvent expansion}
\label{sec: expansion}
We recall the main result from the work of A. Edelman and I. Dumitriu \cite{edelmandumitriu}, which will be our starting point. Let $m\in \mathbb{N}$, $\beta>0$ and choose $0<d\le 1$. Let $a=m\beta /(2d)$. Consider the bi-diagonal matrix
\[B_\beta = \left( \begin{array}{cccc} \chi_{2a} & & & \\
\chi_{\beta(m-1)} & \chi_{2a-\beta} & & \\
& \ddots  & \ddots & \\
& & \chi_\beta &\chi_{2a-\beta(m-1)}
 \end{array}\right).\]
In the equation above, the symbol $\chi_r$, $r>0$ represents a random variable with \emph{chi} distribution with $r$ degrees of freedom, defined by the probability density function
\[\frac{2^{1-\frac{r}{2}}}{\Gamma(r/2)}x^{r-1}e^{-\frac{x^2}{2}}.\]
$\Gamma$ denotes the Gamma function, defined for $x>0$ by
\[\Gamma(x)=\int_0^\infty t^{x-1}e^{-t}\,\mathrm{d}t.\]
The random variables appearing in the matrix $B_\beta$ are all independent. By \cite{edelmandumitriu}, the eigenvalue distribution of the real tridiagonal matrix $\tilde{M}_{\beta,d} = B_\beta  B_\beta^t$ is given by the density function $f_{\beta,a}$ in (\ref{eq: density}). We will be concerned with the scaled $\beta$-Laguerre matrix model, defined by
\[M_{\beta,d} = \frac{d}{\beta m}\tilde{M}_{\beta,d} = \frac{d}{\beta m}B_\beta B_\beta^t.\]
We will determine the behavior of the normalized trace of the resolvent, $s_{\beta,d}(z)$ in a neighborhood of the real axis in the upper half-plane $\mathbb{C}^+ =\{\Im z>0\}$:
\[s_{\beta,d}(z) = \frac{1}{m}\mathrm{tr}\,(M_{\beta,d} -z)^{-1}.\]
The imaginary part $\Im s_{\beta,d}(z)$ is the Poisson integral of the eigenvalue distribution. As such, it is an approximation to the empirical eigenvalue distribution of $M_{\beta,d}$ at scale $\Im z$. To prove Theorem \ref{thm1}, it will suffice to show that, with probability no less than $1-C_{\delta,\kappa,\epsilon,k}m^{-k}$, we have
\[|s_{\beta,d}(z)-s_{MP,d}(z)|<\delta,\]
where
\[s_{MP,d}(z) = \int \frac{1}{x-z} \,\rho_{MP,d}(x)\mathrm{d}x= \frac{z+1-d}{2dz}+\frac{\sqrt{(z-\lambda_-)(z-\lambda_+)}}{2dz}\]
is the Stieltjes transform of the Marchenko-Pastur distribution with parameter $d$, (see \cite{erdosyau5}, Lemma B.1 or \cite{erdosyau6} 7.1, as well as Corollary \ref{cor:ESYY} below).

We begin by writing
\begin{equation}\label{eq: sum}M_{\beta,d} = B_\beta B_{\beta}^t = M_{\infty,d} + \Delta,\end{equation}
where $M_{\infty,d}$ is the matrix
\[\frac{d}{m}\left(\begin{array}{cccc} m/d & \sqrt{m/d}\sqrt{m-1} &  &  \\
\sqrt{m/d}\sqrt{m-1} &m/d+m-2 & \sqrt{m/d-1}\sqrt{m-2} & \\
 & \ddots & &\\
& \sqrt{m/d-m+2}\sqrt{2} & m/d-m+4 & \sqrt{m/d-m+1}\sqrt{1}\\
& & \sqrt{m/d-m+1}\sqrt{1} & m/d-m+2 
 \end{array} \right).\]
The entries of the tridiagonal matrix $\Delta$ are
\begin{align*}
\Delta_{jj} &= \frac{d\chi^2_{2a-(j-1)\beta}-\beta m +d(j-1)\beta}{\beta m}+\frac{d\chi^2_{\beta(m-j+1)}-d\beta (m-j+1)}{\beta m},\\
\Delta_{jj-1} &= d\frac{\chi_{2a-(j-1)}\chi_{\beta(m-j)}-\sqrt{m/d-j-1}\sqrt{m-j}}{m}.
\end{align*}
For large $m$, all the entries of $\Delta$ are simultaneously small in magnitude, with overwhelming probability:
\begin{align}\label{eq: deltasmall}
|\Delta_{jk}| \le m^{-1/2+c}, \quad 1\le j,k \le m.
\end{align}
for any $c>0$. Here and below, we will say an event $E=E(m)$ holds \emph{with overwhelming probability} if, for each $k$, there is a constant $C_k$ such that
\[\mathbb{P}(E^c)\le C_km^{-k}.\]
 This follows readily from the definitions, and properties of the $\chi_r$ and $\chi^2_r$ distributions, which are concentrated around their mean 
\begin{align*}
\mathbb{E}\chi_r &= \sqrt{2}\cdot \frac{\Gamma((r+1)/2)}{\Gamma(r/2)} = \sqrt{r}\cdot (1+O(1/r)),\\
\mathbb{E}\chi^2_r &=r,
\end{align*}
with exponential tails.

\subsection{The matrix $M_{\infty,d}$ and generalized Laguerre polynomials}\label{sec: laguerrepoly}
The spectral theory of the symmetric matrix $M_{\infty,d}$ can be described explicitly in terms of Laguerre polynomials. For $\alpha>-1$, the \emph{generalized Laguerre polynomials} $L^\alpha_k$, $k= 0,1, \ldots$ are orthogonal polynomials with respect to the measure
\[w_\alpha(x)\,\mathrm{d}x=x^\alpha e^{-x}\,\mathrm{d}x\]
on the positive real axis $\mathbb{R}^+ = [0,\infty)$. See \cite[Section 4.5]{bealswong}, \cite[Section 22.7]{astegun} $L^\alpha_k$ is normalized such as to have $L^2(\mathbb{R}^+,w_\alpha\mathrm{d}x)$-norm 1:
\[\int L^\alpha_k(x) L^\alpha_j(x) w_\alpha(x)\,\mathrm{d}x = \delta_{j,k}.\]
The eigenvalues of $M_{\infty,d}$ are the normalized zeros 
\begin{equation}\label{eqn: ev-zero}
\lambda_i=d\cdot l_i/m
\end{equation}
of the $m$th generalized Laguerre polynomial with parameter
\[\alpha = m\cdot\left(\frac{1}{d}-1\right).\] 
There is a complete set of corresponding eigenvectors $v_i$ of the form
\begin{equation}\label{eq: evector}
v_i=\left(\begin{array}{c} L^{m(1/d-1)+1}_{m-1}(l_i)\\ L^{m(1/d-1)+1}_{m-2}(l_i)\\ \vdots \\ L^{m(1/d-1)+1}_0(l_i) \end{array} \right),\end{equation}
for $0\le i \le m-1$. We denote by $u_i$ the normalized eigenvectors
\[u_i=\frac{v_i}{\|v_i\|}.\]
To derive the above facts regarding the eigenvalues and eigenvectors of $M_{\infty,d}$, we start from the three term recurrence relation for the Laguerre polynomials $\tilde{L}^\alpha_k$, normalized to have leading coefficient
\[(-1)^k\frac{1}{k!}.\]
The three term recurrence relation reads (see \cite[Section 4.11, p. 143]{bealswong})
\begin{equation}\label{eqn: threeterm}
x\tilde{L}_k^\alpha(x)=-(k+1)\tilde{L}^\alpha_{k+1}(x)+(2k+\alpha+1)\tilde{L}^\alpha_k(x)-(k+\alpha)\tilde{L}^\alpha_{k-1}(x).
\end{equation}
Letting $k=m-j$ for $j=1,\ldots m-1$ and replacing $\alpha$ by $\alpha+1=m(1/d-1)+1$ in \eqref{eqn: threeterm}, we find:
\begin{equation}\label{eqn: threeterm2}
\begin{split}
x\tilde{L}_{m-j}^{\alpha+1}(x)&=-(m-j+1)\tilde{L}^{\alpha+1}_{m-j+1}(x)+(m/d+m-2j+2)\tilde{L}^{\alpha+1}_{m-j}(x)\\
&\quad -(m/d-j+1)\tilde{L}^{\alpha+1}_{m-j-1}(x).
\end{split}
\end{equation} 
The polynomials $\tilde{L}^\alpha_k$ are related to the orthonormal polynomials $L^\alpha_k$ by 
\begin{equation}\label{eqn: laguerrenormalization}
L^\alpha_{k}(x) = (-1)^k \frac{\sqrt{k!}}{\sqrt{\Gamma(\alpha+k+1)}}\tilde{L}^\alpha_k(x),
\end{equation}
see \cite[Section 4.5]{bealswong}.
Multipliying \eqref{eqn: threeterm2} by 
\[\frac{\sqrt{(m-j)!}}{\sqrt{\Gamma(\alpha+m-j+2)}}.\]
and using \eqref{eqn: laguerrenormalization} and the relation $\Gamma(x+1)=x\Gamma(x)$, we find
\begin{equation}\label{eqn: threeterm3}
\begin{split}
xL_{m-j}^{\alpha+1}(x)&=\sqrt{m-j+1}\sqrt{m/d-j+1}\cdot L^{\alpha+1}_{m-j+1}(x)\\
&\quad +(m/d+m-2j+2)\cdot L^{\alpha+1}_{m-j}(x)\\
&\quad +\sqrt{m-j}\sqrt{m/d-j+1}\cdot L^{\alpha+1}_{m-j-1}(x).
\end{split}
\end{equation}
Letting $x=l_i$, where $l_i$ is one of the $m$ (real) zeros of $L^{\alpha+1}_m(x)$, it follows that 
\begin{equation}\label{eqn: eigenv}
\left(M_{\infty,d}\,v_i\right)_j= \left(\frac{d}{m}l_i \cdot v_i\right)_j
\end{equation}
for $j=2,\ldots, m-1$.
On the other hand, when $j=1$, we use the recurrence relation \cite[(22.7.30)]{astegun}:
\[\tilde{L}_k^{\alpha}=\tilde{L}^{\alpha+1}_k-\tilde{L}^{\alpha+1}_{k-1}\]
in \eqref{eqn: threeterm2} before normalizing. Since $L^\alpha_m(l_i)=0$, equation \eqref{eqn: threeterm3} then simplifies to
\begin{equation}\label{eqn: twoterm1}
l_i\cdot L_{m-1}^{\alpha+1}(l_i)=(m/d)\cdot L^{\alpha+1}_{m-1}(l_i) + \sqrt{m-1}\sqrt{m/d}\cdot L^{\alpha+1}_{m-2}(l_i),
\end{equation}
which yields 
\[\left(M_{\infty,d}\,v_i\right)_1= \left(\frac{d}{m}l_i \cdot v_i\right)_1.\]
Similar reasoning when $j=m$ shows that \eqref{eqn: eigenv} holds for all $j=1,\ldots,m$.

\subsection{The resolvent expansion}
In this section we show that the normalized trace of the resolvent is well approximated by the Stieltjes transform of the Marchenko-Pastur density for $\Im z > m^{-1/4+\epsilon}$, $\epsilon>0$. Specifically, we have the following
\begin{proposition}\label{onequarterprop}
Let $\epsilon, \kappa >0$. For any $z$ such that $\Im z> m^{-1/4+\epsilon}$ and
\[\lambda_-+\kappa <\Re z < \lambda_+-\kappa,\]
we have
\[|s_\beta(z)-s_\infty(z)| = O_{\epsilon,\kappa}(m^{-\epsilon/8})\]
with overwhelming probability.
\end{proposition}
The choice of the scale $m^{-1/4+\epsilon}$ in Proposition \ref{onequarterprop} is somewhat arbritrary, given that the inductive argument of Section \ref{sec: inductive} will gradually improve any initial convergence result to the optimal scale $m^{-1+\epsilon}$. The exponent $\frac{1}{4}-\epsilon$ was chosen because it represents a scale accessible without detailed information on the size of the entries of the resolvent. Using the estimates for Laguerre polynomials in \cite{krasikovbound}, and the method of \cite{krasikovquadratic} to obtain a lower bound for $\|v_i\|$, one can prove sharper bounds than \eqref{eq: offdiag} for $|(M_{\infty,d}-z)^{-1}_{ij}|$ when $|i-j|\le m^{1/4}$. The expansion \eqref{eq: expansion} can then be summed for values of $\Im z$ smaller than $m^{-1/4}$, but still much larger than $m^{-1+\epsilon}$. We will not pursue this here.

The proof of Proposition \ref{onequarterprop} proceeds by a resolvent expansion comparing the trace of the resolvent of $M_{\beta,d}$ to that of the deterministic matrix $M_{\infty,d}$. The precise spectral information available for the matrix $M_{\infty,d}$ allows us to calculate the large $m$ limit of the normalized trace $s_\infty(z)$, and to control the resolvent expansion (\ref{eq: expansion}).
Starting from (\ref{eq: sum}), write:
\[(M_{\beta,d}-z)^{-1}=(M_{\infty,d}-z)^{-1}-(M_{\infty,d}-z)^{-1}\Delta(M_{\beta,d}-z)^{-1}.\]
Upon iteration, this yields
\begin{multline} (M_{\beta,d}-z)^{-1} = (M_{\infty,d}-z)^{-1}\\ +\sum_{k=1}^{l-1}\left(-(M_{\infty,d}-z)^{-1}\Delta\right)^k(M_{\infty,d}-z)^{-1} + (-(M_{\infty,d}-z)^{-1}\Delta)^{l}(M_{\beta,d}-z)^{-1}.
\end{multline}
Taking traces and normalizing:
\begin{multline}\label{eq: expansion} s_{\beta,d}(z)=s_\infty(z) + \frac{1}{m}\sum_{k=1}^{l-1}\operatorname{tr}\left(-(M_{\infty,d}-z)^{-1}\Delta\right)^k(M_{\infty,d}-z)^{-1}\\ + \frac{1}{m}\operatorname{tr}(-(M_{\infty,d}-z)^{-1}\Delta)^{l}(M_{\beta,d}-z)^{-1}.\end{multline}
By the estimate (\ref{eq: approx}) for the normalized trace $s_\infty(z)$ of $(M_{d,\infty}-z)^{-1}$, we can replace the first term on the right with $s_{MP,d}(z)$, introducing an $O_\kappa(m^{-\epsilon})$ error. It suffices to show that the remaining terms also decay at least at an $O(m^{-\epsilon})$ rate as $m\rightarrow \infty$.

For $1\le i,j\le m$, we let
\[R_{ij} = R_{ij}(z) := (M_{\infty,d}-z)^{-1}_{ij},\]
and
\[R_{ij}^{\beta,d}=R_{ij}^{\beta,d}(z):= (M_{\beta,d}-z)^{-1}_{ij}.\]
Expand the $k$th term in of the sum in (\ref{eq: expansion}) to find:
\[(-1)^k\sum_{1\le i_1,\ldots i_{k+1} \le m}R_{i_1i_2}\Delta_{i_2i_2'}R_{i_2',i_3} \cdots R_{i_{k}'i_{k+1}}\Delta_{i_{k+1}i_{k+1}'}R_{i_{k+1}'i_1},\]
where $i_l'=i_l$, $i_l+1$ or $i_l-1$.
Taking absolute values and using (\ref{eq: deltasmall}), this is bounded by
\begin{equation} \label{eq: oneterm}
m^{-(k+2)(1/2-c)}\sum_{1\le i_1,\ldots i_{k+1} \le m}|R_{i_1i_2}||R_{i_2',i_3}| \cdots |R_{i'_{k+1}i_1}|,
\end{equation}
with overwhelming probability.

The next lemma gives an estimate for the entries $R_{ij}$ of the resolvent which is efficient when $i$ and $j$ are widely separated.
\begin{lemma}\label{resolventest}
For any $1\le i, j \le m$, we have
\begin{equation}\label{eq: offdiag}
|R_{ij}(z)| \le C(\Im z)^{-1} \cdot \exp(-c_d (\Im z)\cdot |i-j|). \end{equation}
\end{lemma}
Several results on exponential decay of the resolvent entries could be used to obtain Lemma \ref{resolventest}. We will use a slight variant of the estimate of Combes-Thomas type developed by M. Aizenman in the context of localization for discrete random Schr\"odinger operators:
\begin{lemma}[M. Aizenman, \cite{aizenman}, Lemma II.1]\label{aizlemma}
Let $\Gamma$ be a countable set with a metric $d: \Gamma\times \Gamma \rightarrow [0,\infty)$, and let $H$ be a self-adjoint operator on $\ell^2(\Gamma)$ whose off-diagonal elements are exponentially summable:
\begin{equation}\label{eq: Salpha}
S_\alpha =\sup_x \sum_{y\neq x} |H(x,y)|e^{\alpha d(x,y)}<\infty.
\end{equation}
Then, for energies not in the spectrum of $H$, with
\[\Delta = \operatorname{dist}(z,\sigma(H))>0,\]
we have
\[\Big|(H-z)^{-1}_{xy}\Big| \le \frac{2}{\Delta}\exp\left(-\frac{\alpha\Delta}{\Delta+2S_\alpha}\cdot d(x,y)\right).\]
\end{lemma}
Note that the hypothesis \eqref{eq: Salpha} differs from the one in Lemma II.1 in \cite{aizenman}, but the same proof works also for Lemma \ref{aizlemma} as stated. 

To bound \eqref{eq: oneterm}, we can assume that $i_j=i_j'$. Indeed, the bound \eqref{eq: offdiag} only changes by a constant factor when the index $i$ is changed by a unit.
Consider the sum over $i_1$ in \eqref{eq: oneterm}:
\[\sum^m_{i_1=1}|R_{i_1i_2}||R_{i_{k+1}i_1}|.\]
We split this sum into three regions:
\[\left(\sum_{|i_1-i_2|\ge m^{1/4}}+ \sum_{|i_1-i_2|\le m^{1/4}, |i_1-i_{k+1}|\ge m^{1/4}}+\sum_{|i_1-i_2|\le m^{1/4}, |i_1-i_{k+1}|\le m^{1/4}}\right)|R_{i_1i_2}||R_{i_{k+1}i_1}|.\]
By \eqref{eq: offdiag}, the first two sums are bounded by
\[Cm^{1/2}\exp(-m^\epsilon).\]
We are reduced to considering the sum
\begin{equation}
\label{eq: firstsum}
\sum_{|i_1-i_2|\le m^{1/4}}|R_{i_1i_2}||R_{i_1i_{k+1}}|.
\end{equation}
The second factor in each summand of \eqref{eq: firstsum} is estimated as follows:
\begin{equation}
|R_{i_1i_{k+1}}|=\Big| \sum_{i=1}^m\frac{u_i(i_1)u_i(i_{k+1})}{\lambda_i-z}\Big| \le \left(\sum^m_{i=1}\frac{|u_i(i_1)|^2}{|\lambda_i-z|}\right)^{1/2} \left(\sum^m_{i=1}\frac{|u_i(i_{k+1})|^2}{|\lambda_i-z|}\right)^{1/2}. \label{eq: CS}
\end{equation}
The first factor on the right side of the inequality \eqref{eq: CS} is no greater than $m^{1/8-\epsilon/2}$, since $\|u_i\|=1$ and 
\[\frac{1}{|\lambda_i-z|}\le (\Im z)^{-1}.\]
In summary, \eqref{eq: firstsum} is bounded by
\[m^{1/8-\epsilon/2} \left(\sum_{i=1}\frac{|u_i(i_{k+1})|^2}{|\lambda_i-z|}\right)^{1/2} \cdot \sum_{|i_1-i_2|\le m^{1/4}}|R_{i_1i_2}| \le m^{1/8+1/2-(3/2)\epsilon}\cdot\left(\sum_{i=1}\frac{|u_i(i_{k+1})|^2}{|\lambda_i-z|}\right)^{1/2}.\]
We proceed to sum over $i_2$ in \eqref{eq: oneterm}:
\[\sum_{i_2}|R_{i_2i_3}|\le \sum_{|i_2-i_3|\le m^{1/4} }|R_{i_2i_3}|+C\le Cm^{1/4}\cdot m^{1/4-\epsilon}.\]
We can now repeatedly sum over $i_3, \ldots, i_{k-1}$, using \eqref{eq: offdiag} at every step. Each summation results in an additional factor of $m^{1/2-\epsilon}$. Thus \eqref{eq: oneterm} is bounded by
\begin{equation}\label{eq: sumkminus2}
C^{k-2}m^{1/8-\epsilon/2} \cdot m^{(k-1)/2-(k-1)\epsilon}\sum_{i_k} \left(\sum^m_{i=1}\frac{|u_i(i_k)|^2}{|\lambda_i-z|}\right)^{1/2} \cdot \sum_{|i_k-i_{k+1}|\le m^{1/4}}|R_{i_{k}i_{k+1}}|.
\end{equation}
Using the Cauchy-Schwarz inequality as in \eqref{eq: CS}, we have
\[|R_{i_{k}i_{k+1}}| \le m^{1/8-\epsilon/2}\cdot\left(\sum_{i=1}\frac{|u_i(i_{k+1})|^2}{|\lambda_i-z|}\right)^{1/2}.\]
By \eqref{eq: approx} and Corollary \ref{cor:ESYY},
\[\sum_{i_{k+1}}\sum_{i=1}\frac{|u_i(i_{k+1})|^2}{|\lambda_i-z|} = \sum_{i} \frac{1}{|\lambda_i-z|} \le Cm\log m.\]
Inserting the previous two inequalities into \eqref{eq: sumkminus2}, we find that the sum \eqref{eq: oneterm} is no greater than
\[C^k\log m \cdot m^{(k+2)/2-k\epsilon}.\]
Choosing $c$ smaller than $\epsilon/2$ in \eqref{eq: deltasmall}, this last quantity is $O(m^{-k\epsilon/2})$ with overwhelming probability. As for the final term in \eqref{eq: expansion}, it is bounded by
\[m^{-(k+2)(1/2-c)}\sum_{1\le i_1,\ldots i_k \le m}|R_{i_1i_2}||R_{i_2',i_3}| \cdots |R_{i'_{k-1}i_k}|\cdot m^{1/4-\epsilon},\]
since
\[|R^{\beta,d}_{ij}(z)|\le m^{1/4-\epsilon}\]
for each $i$ and $j$, provided $\Im z > m^{1/4-\epsilon}$. Performing each of the $k$ sums using \eqref{eq: offdiag} as previously, we find that the sum is bounded by
\[C^k m^{-k\epsilon/2}\cdot m^{1/4},\]
with overwhelming probability. Letting $k$ in \eqref{eq: expansion} be larger than $8/\epsilon$, we obtain
\begin{equation}\label{eq: onequarter}
s_\beta(z) =s_\infty(z) +O(m^{-\epsilon/2}).
\end{equation}
with overwhelming probability. Combined with the approximation \eqref{eq: approx}, the relation \eqref{eq: onequarter} implies Proposition \ref{onequarterprop}.

\section{Convergence for $\beta=\infty$}
\label{sec: infty}
In this section, we identify the limit of the resolvent
\begin{equation}\label{eqn: ratio}
s_\infty(z) = \frac{1}{m}\sum_{j=1}^m \frac{1}{\lambda_j-z} = \frac{1}{m}\frac{\left(\tilde{L}^{\alpha}_m(mz/d)\right)'}{\tilde{L}^\alpha_m(mz/d)}
\end{equation}
for $z$ as close as $m^{-1/2+\epsilon}$ to the limiting spectrum, but away from its edges. The eigenvalues $\lambda_i$ are given by \eqref{eqn: ev-zero}. In \cite{sosoewong}, the corresponding quantity for Hermite polynomials was estimated using asymptotics for these orthogonal polynomials in the complex plane. Here, we use a differential equation for $s_\infty(z)$ derived from the ODE satisfied by Laguerre polynomials (\cite[(22.6.15)]{astegun}, \cite[Section 4.5, (4.5.1)]{bealswong}: 
\begin{equation}
\label{eqn: 2ndorderode}
x(\tilde{L}^\alpha_m)''(x)+(\alpha+1-x)(\tilde{L}_m^\alpha)'(x)+m\tilde{L}_m^\alpha(x)=0,
\end{equation}
with $\alpha= m(1/d-1)$.
Differentiating the ratio on the right of \eqref{eqn: ratio} and using \eqref{eqn: 2ndorderode}, we find an equation for $s_\infty(z)$ (see \cite[Section 2.2.2, Eqn. (20)]{edelmandumitriu2}),
\begin{equation}
zd(s_\infty(z))^2 +s_\infty(z)\left(d-1+z\right)+1+\frac{d}{m}s_\infty(z)+\frac{dz}{m}s_\infty(z)'=0.
\end{equation}
To solve the equation approximately, we treat the final two terms
\[\mathbf{\varepsilon} =  \frac{d}{m}s_\infty(z)+\frac{dz}{m}s_\infty(z)'\]
as error terms, and use the rough estimates
\[|s_\infty(z)| \le m^{1/2-\epsilon},
|s_\infty'(z)| \le m^{1-2\epsilon}.\]
These follow at once from $|\lambda_i-z|\ge m^{1/2-\epsilon}$, so that $\mathbf{\varepsilon} = O(m^{-\epsilon})$. 
For $z$ in the region 
\[\{\Im z>m^{-1/2+\epsilon}, |z-\lambda_{\pm}|>\kappa, |z|\le 4\},\] 
we find the solutions
\[s^\pm(z) =\frac{z+d-1\pm \sqrt{(z-\lambda_+)(z-\lambda_-)}}{2dz}+O_{\kappa}(m^{-\epsilon}).\]
The solution $s^-$ satisfies $\Im s^->0$ for $\Im z>0$, and we conclude that the right side of the equation is an approximation for $s_{\beta,d}$:
\begin{equation}\label{eq: approx}
s_\infty(z) =s_{MP,d}(z)+O_{\kappa}(m^{-\epsilon}).
\end{equation}
We expect that the error term can be replaced by $O(1/m)$, but \eqref{eq: approx} is sufficient for our purposes.
\section{Inductive argument}
\label{sec: inductive}
In this section, we improve the convergence of the density on short scales from the level $m^{-1/4+\epsilon}$ to the optimal level of $m^{-1+\epsilon}$ by an inductive argument.

\begin{theorem} \label{thm:main}
Let $s(z)=s_{\beta,d}(z)$ be the Stieltjes transform of the measure induced by the eigenvalues of the normalized $\beta$-Laguerre ensemble matrix $M_{\beta,d}$. Let $s_{MP,d}(z)$ be the Stieltjes transform of the Marchenko-Pastur distribution. Then, with overwhelming probability,
\begin{equation}
\sup_{z\in D_{\epsilon,\kappa}}|s(z)-s_{MP,d}(z)| = o(1)
\end{equation}
where the domain $D$ is defined as
\[D_{\epsilon,\kappa} := \{z: \Im z > m^{-1+\epsilon}, \lambda_-+\kappa < \Re z < \lambda_+-\kappa, |z| \leq 4\}.\]
\end{theorem}

To prove Theorem \ref{thm:main}, we need three facts about the tridiagonal models and Stieltjes transforms. The first can be found in \cite{dumitriuth}:
\begin{proposition} \label{prop:indep}
One can diagonalize 
\[M_{\beta,d} = Q\Lambda Q^*\]
 such that the first row of $Q$ is independent of $\Lambda$ and consists of independent entries with $\chi_\beta$ distribution, normalized to unit norm.
\end{proposition}

The next corollary establishes the link between control of Stieltjes transform and control of the eigenvalue density. See for example \cite{erdosyau5}:
\begin{corollary} \label{cor:ESYY}
Let $z = E + i\eta$ and $m^{-1+\epsilon} \leq \eta \leq \frac{1}{2}E$, $\frac{1}{2}\lambda_- \leq E \leq 10$ and $\epsilon,\tau > 0$.  Suppose that one has the bound
\begin{displaymath}
|s_{\beta,d,m}(z) - s_{MP,d}(z)| \leq \tau
\end{displaymath}
with overwhelming probability for all such $z$.  Then for any interval $I$ in $[\lambda_-+\epsilon,\lambda_+-\epsilon]$ with $|I|\geq \eta$, one has
\begin{displaymath}
\left|N_I - n\int_I \rho_{MP,d}(y) \,\mathrm{d}y\right| \le C_\epsilon \tau n|I|
\end{displaymath}
with overwhelming probability, where $N_I$ denotes the number of eigenvalues in $I$.
\end{corollary}

Finally, we shall need the following standard result. A proof can be found in \cite{sosoewong}:
\begin{lemma} \label{lem:absest}
Suppose \emph{the Marchenko-Pastur distribution holds at level $m^a$} for some $-1 < a \le 0$, that is, for any $c>0$, 
\[|s(z)-s_{MP,d}(z)|<c\]
for sufficiently large $m$, with overwhelming probability.

Then we have
\begin{equation}
\frac{1}{m}\sum_j \frac{1}{|\lambda_j-z|^2} \leq C (\Im z)^{-2}m^a \log m
\end{equation}
for any $z$ such that $m^{-1} < \Im z < m^a$.
\end{lemma}

The proof of Theorem \ref{thm:main} is a combination of the following propositions:
\begin{proposition} \label{prop:induct1}
Suppose the Marchenko-Pastur law holds at level $m^a$ for some $-1 < a \le 0$, and that we have
\begin{equation}\label{eq: R11approxbigoh}
|R^{\beta,d}_{11}(z) - s_{MP,d}(z)| = O(1)
\end{equation}
for $\Im z> m^a$ with overwhelming probability. Then, for any $\delta>0$, 
\begin{equation}
|R^{\beta,d}_{11}(z) - s_{MP,d}(z)| = o(1)
\end{equation}
holds for $z$ such that $\Im z > m^{(a-1)/2+\delta}$ with overwhelming probability.
\end{proposition}
 For a given sequence $X_n$ of random variables, we write
\[X_n = o(1)\]
\emph{with overwhelming probability} if, for every $c>0$, the event $\{|X_n|\le c\}$ has overwhelming probability.
\begin{proposition} \label{prop:induct2}
Suppose that with overwhelming probability, $|R^{\beta,d}_{11}(z) - s_{MP,d}(z)| = o(1)$ for $z$ such that  $\Im z > m^a$ for some $-1 < a \le 0$, then we have an improved Marchenko-Pastur law, that is:
\begin{equation}
\left|\frac{1}{m} \sum_j \frac{1}{\lambda_j-z} - s_{MP,d}(z)\right| = o(1)
\end{equation}
for $z$ such that $\Im z > m^{(a-1)/2+\delta}$ for any $\delta > 0$ with overwhelming probability.
\end{proposition}

\begin{proof} [Proof of Theorem \ref{thm1}]
The condition \eqref{eq: R11approxbigoh} holds trivially if $\Im z \ge 1$. The result follows by repeatedly applying Propositions \ref{prop:induct1} and \ref{prop:induct2}.
\end{proof}

What remains is the proof of the two propositions above:
\begin{proof} [Proof of Proposition \ref{prop:induct1}]
By Schur's complement, we have the following relation
\begin{equation} \label{eqn:R11}
R^{\beta,d}_{11}(z) = \frac{1}{a^2_{11}-z-a^2_{11}a^2_{21}\frac{\widehat{R}_{11}^{\beta,d}}{1+a^2_{21}\widehat{R}^{\beta,d}_{11}}}.
\end{equation}
Here $a_{11}$ denotes the normalized $(1,1)$-entry of the bidiagonal matrix $B_\beta$, $a_{2,1}$ the normalized $(2,1)$-entry of $B_\beta$, $\widehat{R}^{\beta,d}$ is the resolvent of $\widehat{M_\beta}:=\frac{d}{\beta m}\widehat{B_\beta}\widehat{B_\beta}^t$ where $\widehat{B_\beta}$ is formed by removing the first row and column of $B_\beta$.  We remark that $a^2_{11}$ is distributed as $\frac{d}{\beta m}\chi_{2a}$ and $a^2_{21}$ distributed like $\frac{d}{\beta m}\chi_{\beta(m-1)}$. Lastly, $a_{11}$, $a_{21}$ and $\widehat{R}$ are independent. 

By the argument in Section \ref{sec: infty}, we need only show that $R^{\beta,d}_{11}$ satisfies the approximate functional equality
\begin{equation} \label{eqn:func}
R^{\beta,d}_{11}(z) + \frac{1}{d-1+z+zd\cdot R^{\beta,d}_{11}(z)} = o(1),
\end{equation}
with overwhelming probability for $z$ such that $\Re z \in [\lambda_-+\kappa, \lambda_+ - \kappa]$ and $\Im z \geq m^{(a-1)/2+\delta}$, with $\delta > 0$ arbitrary.
We first note that, due to the restriction on $\Re z$ (in particular, $\Re z > \kappa$ since $\lambda_- > 0$), equation \eqref{eqn:func} is equivalent to:
\begin{equation} \label{eqn:func2}
zd \cdot (R_{11}^{\beta,d}(z))^2 + (d-1+z)\cdot R_{11}^{\beta,d}(z) + 1 = o(1)
\end{equation}
where from now on we suppress the superscript $\beta,d$.

Rewriting equation \eqref{eqn:R11}, we have
\begin{equation}
z a^2_{21}\cdot \widehat{R_{11}}^{\beta,d}R_{11}^{\beta,d} + zR_{11}^{\beta,d} +a^2_{21}\widehat{R_{11}}^{\beta,d} - a^2_{11}R_{11}^{\beta,d} + 1 = 0.
\end{equation}
We can further rewrite the above as
\begin{multline}
zd\cdot \frac{m-1}{m}\cdot \widehat{R_{11}}^{\beta,d}R_{11}^{\beta,d} + zR_{11}^{\beta,d} + d\frac{m-1}{m}\cdot\widehat{R_{11}}^{\beta,d} - R_{11}^{\beta,d} + 1\\
 + \left(z\cdot\widehat{R_{11}}^{\beta,d}R_{11}^{\beta,d}+\widehat{R_{11}}^{\beta,d}\right)\left(a^2_{21}-d\frac{m-1}{m}\right) - (a^2_{11}-1)\cdot R_{11}^{\beta,d} = 0.
\end{multline}

It suffices, therefore, to show that
\begin{equation} \label{eq:suf1}
\frac{m-1}{m}\cdot \widehat{R_{11}}^{\beta,d} = R_{11}^{\beta,d} + o(1)
\end{equation}
and
\begin{equation} \label{eq:suf2}
 \left(z\cdot\widehat{R_{11}}^{\beta,d}R_{11}^{\beta,d}+\widehat{R_{11}}^{\beta,d}\right)\left(a^2_{21}-d\frac{m-1}{m}\right) - (a^2_{11}-1)\cdot R_{11}^{\beta,d} = o(1)
\end{equation}
for the set of $z$ we are interested in.

We first note that $a^2_{21}$ has distribution $\frac{d}{\beta m} \chi^2_{\beta(m-1)}$ and so 
\[\Big|a^2_{21}-d\frac{m-1}{m}\Big| \le C m^{-1/2+\delta}\]
with overwhelming probability for all $\delta > 0$.  Similarly, $a^2_{11}$ has the distribution of $\frac{d}{\beta m} \chi^2_{\beta m/d}$ and so 
\[|a^2_{11}-1| \le C m^{-1/2+\delta}\]
with overwhelming probability for all $\delta > 0$. 
The assumption of the proposition implies that for all $w$ with $\Im w> m^a$,  we have
\[|R_{11}^{\beta,d}(w)| \leq C,\] 
for some constant $C$. By \eqref{eq:suf1}, which will be proved independently below, we also have 
\[|\widehat{R_{11}}^{\beta,d}(w)|\le C.\] 
The inequality
\[|R_{11}^{\beta,d}(z)| \leq \frac{w}{z} |R_{11}^{\beta,d}(w)|,\]
for $\Im z \leq \Im w$, now implies that approximation \eqref{eq:suf2} holds with overwhelming probability whenever $\Im z > m^{(a-1)/2+\delta}$.

We turn our attention to \eqref{eq:suf1}.  Firstly, we write
\[\widehat{R}_{11}^{\beta,d} = \sum_j \frac{\widehat{q}_j^2}{\widehat{\lambda}_j-z}, \]
and similarly
\[R_{11}^{\beta,d} = \sum_j \frac{q_j^2}{\lambda_j-z}, \]
where $(\widehat{q}_1,\hdots\widehat{q}_{m-1})$ is the first row of the eigenvectors for $\widehat{M_\beta}$, $\widehat{\lambda}_j$ are the eigenvalues and $(q_1,\hdots,q_{m})$ is the first row of eigenvectors for $M_\beta$, and $\lambda_j$ the eigenvalues.

We write
\begin{equation} \label{eqn:compare}
\begin{split}
R_{11}^{\beta,d}-\frac{m-1}{m}\cdot \widehat{R}_{11} &= \sum_j\frac{q_j^2}{\lambda_j-z}-\mathbb{E}_q\sum_j\frac{q_j^2}{\lambda_j-z}\\  
&+ \mathbb{E}_q\sum_j\frac{q_j^2}{\lambda_j-z} - \frac{m-1}{m}\cdot\mathbb{E}_{\widehat{q}}\sum_j\frac{\widehat{q}_j^2}{\widehat{\lambda}_j-z}\\
&+ \frac{m-1}{m}\cdot \left(\mathbb{E}_{\widehat{q}}\sum_j\frac{\widehat{q}_j^2}{\widehat{\lambda}_j-z} - \sum_j\frac{\widehat{q}_j^2}{\widehat{\lambda}_j-z}\right),
\end{split}
\end{equation}
where $\mathbb{E}_X$ denotes the expectation with respect to the random variables $X$.

By Proposition \ref{prop:indep}, $q$ and $\lambda_j$ are independent and so are $\hat{q}$ and $\hat{\lambda}_j$. We write the first term in \eqref{eqn:compare} as
\[\sum_j\frac{q_j^2}{\lambda_j-z}-\mathbb{E}_q\sum_j\frac{q_j^2}{\lambda_j-z} = \sum_j\frac{q_j^2-\frac{1}{m}}{\lambda_j-z}.\]
Since $q$ and $\lambda_j$ are independent, we can condition on $\lambda_j$ and apply Proposition \ref{mcdiarmid} below to the sum to conclude that the right side of the previous equation is bounded with overwhelming probability by 
\[\frac{1}{m^{1-c}}\left(\sum_j\frac{1}{|\lambda_j-z|^2}\right)^{1/2}.\] By Lemma \ref{lem:absest}, the latter is bounded by $C (\Im z)^{-1}m^{-1/2+a/2+c} \log m$ for some constant $C$ and any $c > 0$. The fifth and sixth terms of equation (\ref{eqn:compare}) are similarly bounded. To deal with the middle two terms, we use the interlacing property of the eigenvalues of a matrix and its minor, and then the interlacing property of the eigenvalues of a matrix and a rank $1$ perturbation, to obtain that these terms are bounded by 
\[\frac{1}{m\Im z}.\]

We end the proof of Proposition \ref{prop:induct1} with the following simple variant of McDiarmid's inequality, a proof of which can be found in the appendix of \cite{sosoewong}:
\begin{proposition}\label{mcdiarmid}
Let $X_1,\hdots,X_m$ be independent subgaussian random variables. Let $\Omega\subset \mathbb{R}^n$ be such that
\[(X_1,\ldots, X_m) \in \Omega\]
with overwhelming probabibility. Let  $F$ be a real function of $m$ real variables such that if $x_1,\hdots,x_m,\tilde{x_i} \in \Omega$, then  
\begin{displaymath}
|F(x_1,\hdots,x_n) - F(x_1,\hdots,x_{i-1},\tilde{x_i},x_{i+1},\hdots,x_n)| \leq c_i
\end{displaymath}
for all $1\leq i \leq m$, and outside of $\Omega$, $F$ is bounded by a polynomial in $m$.  Then for any $\lambda > 0$, one has
\begin{displaymath}
\mathbb{P}(|F(X)-\mathbb{E}(F(X))| \geq \lambda \sigma) \leq C \exp(-c\lambda^2)
\end{displaymath}
for some absolute constants $C$,$c > 0$, and $\sigma = \sum_{i=1}^mc_i^2$.
\end{proposition} 
\end{proof}

\begin{proof} [Proof of Proposition \ref{prop:induct2}]
The proof is similar to that of Proposition \ref{prop:induct1}.  By the assumption, it suffices to establish that
\begin{displaymath}
\left|\frac{1}{m}\sum_j\frac{1}{\lambda_j-z}-R_{11}^{\beta,d}(z)\right| = o(1)
\end{displaymath}
with overwhelming probability.  The difference inside the absolute value sign is $\sum_j\frac{\frac{1}{m}-q^2_j}{\lambda_j-z}$. The statement now follows by another concentration argument and Lemma \ref{lem:absest}.
\end{proof}

\section{Some extensions}\label{sec: extensions}
The resolvent expansion in Section \ref{sec: expansion} allows for a quick proof of convergence of the eigenvalue distribution for tridiagonal matrices with independent entries ``close to'' a Jacobi matrix whose limiting spectral density is known. Consider a sequence of random tridiagonal matrices given by
\begin{equation}\label{eq: matrix}
A_n = \left( \begin{array}{ccccc}
a_n & b_{n-1} & & & \\
b_{n-1} & a_{n-1} & b_{n-2} & & \\
& & \ddots & & \\
& &  b_1   & a_1 & b_0\\
& & & b_0 & a_0
\end{array}
\right),
\end{equation}
where $a_j$, $b_j$, $1\le j \le n$. Suppose that
\begin{align} \label{eq: close}
\frac{b_j-\bar{b}_j}{n^\alpha} = o(1),\\
\frac{a_j-\bar{a}_j}{n^\alpha} = o(1)
\end{align}
uniformly in $1\le j\le n$ with overwhelming probabilty, for some $\alpha >0$, where
\begin{align*}
\bar{b}_j = \mathbb{E}b_j, \\
\bar{a}_j =\mathbb{E}a_j.
\end{align*}
If the deterministic tridiagonal matrix $\frac{\mathbb{E}A_n}{n^\alpha}$ has a limiting spectral density $\mu$, that is, if the spectral measures $\mu_n$ of $n^{-\alpha}\cdot \mathbb{E}A_n$ converge weakly to some measure $\mu$, then
\[\frac{1}{n}\operatorname{tr}\left(n^{-\alpha}\mathbb{E}A_n-z\right)^{-1} = \int\frac{1}{z-\lambda}\mu(\mathrm{d}\lambda) +o(1)\]
for $z$ in any compact subset $D$ of $\{\Im z> 0\}$. 
Using Lemma \ref{resolventest} as in the proof of Proposition \ref{onequarterprop} to perform a one-step resolvent rexpansion, we have, for $z\in D$:
\begin{equation}
\label{eq: littleoh}
\frac{1}{n}\operatorname{tr}\left(n^{-\alpha}A_n-z\right)^{-1} - \int\frac{1}{z-\lambda}\mu(\mathrm{d}\lambda = o(1)
\end{equation}
with overwhelming probability. Thus the eigenvalue distribution of $n^{-\alpha}A_n$ almost surely converges weakly to $\mu$.

 As an application, consider the following example from \cite{popescu}. Consider the Jacobi matrix
\[
J_n = \left( \begin{array}{ccccc}
0 & (n-1)^\alpha & & & \\
(n-1)^{\alpha} & 0 & (n-2)^\alpha & & \\
&(n-2)^\alpha & \ddots &  & \\
& &  &          & 2^\alpha\\
& & 2^\alpha   & 0 & 1\\
& & & 1 & 0
\end{array}\right).
\]
The moments of the rescaled matrix $n^{-\alpha}J_n$ converge to those of the \emph{Nevai-Ullman distribution} \cite{ullman}, \cite[Section 5]{krasovsky}:
\[\frac{1}{n}\operatorname{tr}(n^{-\alpha}J_n)^k\rightarrow \mu_k,\]
\begin{align*}
\mu_k &= \int x^k\,\nu_\alpha(\mathrm{d}x),\\
\nu_\alpha(\mathrm{d}x) &= \mathbf{1}_{[-2,2]}(x)\frac{1}{\alpha \pi}\int_{|x|/2}^2\frac{t^{-1+1/\alpha}}{\sqrt{4-t^2}}\,\mathrm{d}t.
\end{align*}
By a standard density argument, 
\[\frac{1}{n}\operatorname{tr} (n^{-\alpha}J_n-z)^{-1} = \int \frac{1}{\lambda-z}\,\nu_{\alpha}(\mathrm{d}\lambda)+o(1).\]
If we let
\begin{align*}
b_j &= (J_n)_{jj-1} + n^\beta X_j,\\
a_j &= (J_n)_{jj}+ n^\beta Y_j,
\end{align*}
where $X_j$, $1 \le j \le n-1$ and $Y_j$, $1\le j \le n$, are mean zero random variables whose $k$th moment is bounded uniformly in $j$ for all $k$, then for any $\beta<\alpha$ \eqref{eq: littleoh} holds for the matrix $n^{-\alpha}A_n$.
If we instead require that 
\begin{align}\label{eq: popescuhyp}
\mathbb{E}\left(\frac{b_n}{n^{\alpha}}\right)^k\rightarrow 1\\
\sup_n \mathbb{E}|a_n|^k <\infty.\nonumber
\end{align}
for $k=1,2$, then for any $c>0$, we have
\[\mathbb{P}\left(\left|\frac{b_j-1}{n^\alpha}\right|\ge c\right) \le \frac{1}{c^2}\mathbb{E}\left|\frac{b_j-1}{n^\alpha}\right|^2=o(1),\]
uniformly in $j$, and a similar bound for $a_j$. The convergence \eqref{eq: littleoh} with $\mu = \nu_\alpha$ holds in probability. Popescu \cite{popescu} assumes \eqref{eq: popescuhyp} holds for all $k$ and obtains convergence of the moments of $n^{-\alpha}A_n$ as well as almost sure convergence to the limiting distribution. Note that \cite{popescu} contains more general results that apply also to cases not easily accessible by our method.

Given more precise information on the size of the entries of $A_n$ or the rate of convergence of the Stieltjes transform for the deterministic model, one can improve on the above. To give a simple example, we introduce a ``positive temperature'' version of the Jacobi matrices associated to the orthogonal polynomials $p_n(x)$, $n\ge0$, for the measures 
\begin{equation} \label{eq: opmeasure}
 e^{-x^{2m}}\mathrm{d}x.
\end{equation}
Define the density 
\[g_{n,m}(x) = \mathbf{1}_{[0,\infty)}\frac{x^{n-1}e^{-x^{2m}}}{\frac{1}{2m}{\Gamma(\frac{n}{2m})}}.\]
Let the random variable $X_j$ have distribution given by $g_{j,m}$, with all the $X_j$ independent. We have
\[\mathbb{E} X_j = \frac{\Gamma\left(\frac{j+1}{2m}\right)}{\Gamma\left(\frac{j}{2m}\right)}=\frac{2^{-1/2m}}{e^{1/2m}}\left(\frac{j+1}{m}\right)^{1/2m}\cdot \left(1+\frac{1}{j}\right)^{j/2m}+O(1/j),\]
and, by a straightforward computation:
\[\mathbb{P}\left(|X_j-\mathbb{E}X_j|\ge n^\epsilon\right)\le Ce^{-n^\epsilon},\]
uniformly for $1\le j \le n$.
Consider the matrix \eqref{eq: matrix}, with entries
\begin{align*}
b_j &= \frac{e^{1/2m}}{2^{1+1/2m}}\cdot\left(1+\frac{1}{j}\right)^{j/2m} X_{j},\\
a_j &= N(0,1).
\end{align*}
Let $J_n=(m\alpha_m/n)^{1/(2m)}\bar{A}_n$, with $\bar{A}_n=\mathbb{E}A_n$, and
\[\alpha_m =\prod_{j=1}^m \frac{2j-1}{2j}.\]
The matrix $J_n$ approximates a Jacobi matrix associated with the orthogonal polynomials for the measure \eqref{eq: opmeasure}. Indeed, by \cite[Eq. (2.33)]{deiftetal}, we have
\[\tilde{b}_n=\frac{1}{2}\left(\frac{n+1}{\alpha_mm}\right)^{1/2m}+O(1/n^{2-1/2m}).\]
Here $\tilde{b}_n$ is the coefficient in the three-term recurrence relation:
\[xp_n(x)=\tilde{b}_{n+1}p_{n+1}(x)+\tilde{a}_np_n(x)+\tilde{b}_{n-1}p_{n-1}(x).\]
Note that $\tilde{a}_n=0$. The eigenvalues of the matrix
\[
\tilde{J}_n = \left(\frac{m\alpha_m}{n}\right)^{1/2m}\cdot \left( \begin{array}{ccccc}
0  & \tilde{b}_{n-1} & & & \\
\tilde{b}_{n-1} & 0 & \tilde{b}_{n-2} & & \\
&\tilde{b}_{n-2} & \ddots &  & \\
& &  &    \tilde{b}_1      & \\
& & \tilde{b}_1   & 0 & \tilde{b}_0\\
& & & \tilde{b}_0 & 0
\end{array}\right),
\]
are the (rescaled) zeros of the $n$th orthogonal polynomial with respect to \eqref{eq: opmeasure}. Their limiting density can be explicitly computed; see \cite[Eq. (2.4)]{deiftetal}. It has the form
\[\mu(\mathrm{d}x) = \frac{1}{2\pi}\sqrt{1-x^2}h_m(x)\mathbf{1}_{[-1,1]}(x)\,\mathrm{d}x,\]
where $h_m(x)$ is a polynomial of degree $2m-2$.
A calculation as in \cite{sosoewong} using the Riemann-Hilbert asymptotics in \cite{deiftetal} shows that the normalized trace of $(\tilde{J}_n-z)^{-1}$ approximates the Stieltjes transform of the limiting density $\mu(\mathrm{d}x)$ with precision $O(1/n)$ for $\Im z> n^{-1+\epsilon}$ and $z$ away from $\pm 1$. By the resolvent expansion argument of Section \ref{sec: expansion} and Corollary \ref{cor:ESYY}, convergence of the empirical eigenvalue distribution for $(\alpha_mn)^{-1/2m}A_n$  holds for intervals of size $|I|\ge n^{-1/4m+\epsilon}$ strictly inside the bulk of the limiting density.

\end{document}